\documentclass[english,a4paper]{amsart}
\usepackage{hyperref}

\hypersetup{
	colorlinks=true,
    linkcolor=black,
    citecolor=black,
    filecolor=black,
    urlcolor=black,
}

\usepackage{Definitions}
\usepackage{Environments}
\usepackage{PageSetup}
\usepackage{tikz-cd}

\newcommand{\MSp}{MSp}
\newcommand{\MSO}{MSO}

\newcommand{\MU}{MU}


\begin{document}
	\title{Nilpotence in the Symplectic Bordism Ring}
	\author{Justin Noel}
	\address{University of Regensburg\\
	NWF I - Mathematik; Regensburg, Germany}
	\email{justin.noel@mathematik.uni-regensburg.de}
	\urladdr{http://nullplug.org}
	\thanks{The author was partially supported by CRC 1085 - Higher Invariants, Regensburg.}

\date{\today}

\begin{abstract}
	We prove a folklore result which gives a sequence of necessary and sufficient conditions for a stably symplectic manifold to define a nilpotent element in the symplectic bordism ring. 
\end{abstract}

\maketitle

The purpose of this note is to prove the following folklore result:
\begin{thm}\label{thm:vanishing}\ 
	Let $M$ be a stably symplectic manifold. Then the following are equivalent:
	\begin{enumerate}
		\item\label{it:nil} For all sufficiently large $n$, $M^{\times n}$, the $n$-fold cartesian product of $M$ with itself bounds a stably symplectic manifold. 
		\item\label{it:tor} There exists a positive integer $n$, such that $\coprod_n M$, the $n$-fold disjoint union of $M$ with itself bounds a stably symplectic manifold. 
		\item\label{it:pontryagin} The Pontryagin numbers of $M$ vanish.
		\item\label{it:bu} $M$ bounds a stably complex-manifold.
		\item\label{it:bso} $M$ bounds an oriented manifold.
	\end{enumerate}
\end{thm}

The author would like to thank Nigel Ray for politely pointing out that the identification of the nilpotent symplectic bordism classes has been well-known to experts since 1986 using the nilpotence theorem of \cite{DHS88}. Indeed, this method of identifying the nilpotent classes is mentioned in the introductions of \cite{GoR92,BoK94} and we will recall it in \Cref{rem:nige}. The author makes no claims of originality for \Cref{thm:vanishing}.

We now make the claims in \Cref{thm:vanishing} precise. In this paper all manifolds will be smooth and compact.  A manifold $M$ is called stably symplectic (resp.~stably complex, oriented) if the map \[ M\rightarrow BO\] classifying the stable tangent bundle of $M$ has a designated lift to $BSp$ (resp.~$BU$, $BSO$) in the homotopy category. Here we are regarding $BSp$, the classifying space of the infinite dimensional symplectic group, to be a space over $BO$ via the standard forgetful map. This map is induced by the group homomorphism which takes a unitary quaternionic matrix to its underlying orthogonal matrix. A similar construction holds for $BU$ and the forgetful map $BSp\rightarrow BO$ factors through the forgetful map from $BU$. We note that this notion is weaker than the standard geometric notion of a symplectic structure; in particular, we have no restrictions on the dimension of $M$. We say that a manifold $M$ bounds a stably symplectic (resp.~stably complex, oriented) manifold if there is another stably symplectic (resp.~stably complex, oriented) manifold $W$ such that the boundary of $W$ is $M$ and the stable tangential structure on $M$ is the restriction of the stable tangential structure on $W$ to the boundary. 

The bordism classes of stably symplectic manifolds form a graded ring which can be identified, via the Pontryagin-Thom construction, with $\pi_* \MSp$, the homotopy groups of the Thom spectrum $\MSp$.  Using the machinery of \cite{ABGHR08} this Thom spectrum can be obtained by applying the Thom construction to the composite map \[BSp\xrightarrow{i} BO\xrightarrow{j} BGL_1 S.\] Here $i$ is the forgetful map above. Since the forgetful map respects direct sums and the direct sum operation induces the infinite loop space structures on $BSp$ and $BO$, one can show that $i$ is a map of infinite loop spaces. The space $BGL_1 S$ is the classifying space for stable spherical fibrations and the map $j$ is a delooping of the classical $j$-homomorphism which is known to be an infinite loop map.  Since both $i$ and $j$ are infinite loop maps, $\MSp$ has an induced $E_\infty$-ring structure \cite[\S IX.7]{LMS86}. The induced multiplicative structure on  $\pi_* \MSp$ agrees with the bordism ring structure. Under this isomorphism, addition  corresponds to the disjoint union of manifolds and multiplication corresponds to the cartesian product of manifolds. 

We should note that we only have partial, albeit extensive, information about $\pi_*\MSp$. After inverting 2, Novikov showed that this is a polynomial algebra on generators in each positive degree divisible by 4 \cite{Nov62}, but the two primary groups are only known through a range. Liulevicius, Nigel Ray, and then Kochman calculated these groups through degree 6, 19, and 100, respectively \cite{Liu64, Ray71, Koc93}. Here are the first 15 groups:
\vspace{5pt}
\begin{center}
\begin{tabular}{ | c | c | c | c | c | c | c | c | c | c | c | c | c | c | c |}
	\hline 
	0 & 1 & 2 & 3 & 4 & 5 & 6 & 7 & 8 & 9 & 10 & 11 & 12 & 13 & 14 \\
	\hline
	$\bZ$ & $\bZ/2$ & $\bZ/2$ & 0 & $\bZ$ & $\bZ/2$ & $\bZ/2$ & 0 & $\bZ^2$ & $\bZ/2$ & $\bZ/2^2$ & 0 & $\bZ^3$ & $\bZ/2^2$ & $\bZ/2^2$\\
	\hline
\end{tabular}
\end{center}
\vspace{5pt}
There are a lot of non-trivial products, in particular the class $\eta\in \pi_1 \MSp$ acts non-trivially on many of the classes above and the square of the generator in degree 4 is 4 times a generator in degree 8. By \cite[Thm.~B]{BoK94} the 2-torsion in $\pi_*\MSp$ is unbounded. Note that it follows from \Cref{thm:vanishing} that all of the non-nilpotent classes are in degrees divisible by 4 and that these classes are detected in Novikov's calculation.

Now let $[M]\in \pi_* \MSp$ be the cobordism class of a stably symplectic manifold $M$. We will prove \Cref{thm:vanishing} by showing that all of the stated conditions are equivalent to the following property:
\begin{equation*}
	[M]\in \ker(\pi_* \MSp\xrightarrow{h} H_*(\MSp;\bZ))
\end{equation*}
where $h$ is the integral Hurewicz map (which is a map of graded rings).  Since symplectic bundles are oriented we have a Thom isomorphism of graded rings $H_*(\MSp;\bZ)\cong H_*(BSp;\bZ)$. 

To understand this ring and the behavior of the Hurewicz map, we now recall some well-known results from the theory of characteristic classes: 
\begin{prop}\label{prop:char-classes}\ \\
	\begin{enumerate}
		\item\label{it:integral} $H_*(BSp;\bZ)\cong \bZ[x_{4i}]_{i\geq 1}$, where $|x_{4i}|=4i$.
		\item\label{it:rat-bso} The composite of the forgetful maps \[H_*(BSp;\bQ)\cong \bQ[x_{4i}]_{i\geq 1}\rightarrow H_*(BU;\bQ)\rightarrow H_*(BSO;\bQ)\cong \bQ[y_{4i}]_{i\geq 1}\] is an isomorphism of graded rings. 
		\item\label{it:rat-bu} The forgetful map \[H_*(BSp;\bQ)\rightarrow H_*(BU;\bQ)\cong \bQ[b_{2i}]_{i\geq 1}\] is an injection.
	\end{enumerate}
\end{prop}
\begin{proof}
	It is well known that $H^*(BSp;\bZ)$ is a polynomial algebra in the Pontryagin classes $\{P_{4i}\}_{i\geq 1}$ and is abstractly isomorphic as a Hopf algebra to $H^{*/2}(BU;\bZ)$. The identification of the coalgebra structure is equivalent to the standard identity \[P_{4n}(\xi_1\oplus \xi_2)=\sum_{1\leq i\leq n} P_{4i} (\xi_1)\cup P_{4(n-i)}(\xi_2)\] where $\xi_1$ and $\xi_2$ are two symplectic vector bundles. Since the cohomology ring is torsion free and finitely generated in each degree we see that $H_*(BSp;\bZ)$ is abstractly isomorphic to $H_{*/2}(BU;\bZ)$ as Hopf algebras. 

	The second claim follows by duality from the well known, and easily proven, dual claim in cohomology and the third claim follows from the second. 
\end{proof}
\begin{center}
\begin{figure}
\[
	 	\begin{tikzcd}
	 		\pi_*\MSp \rar \dar &  H_*(\MSp;\bZ)  \dar[hook]\rar[hook]  & H_*(\MSp;\bQ)\dar[hook]\\
	 		 \pi_*\MU \rar[hook] \dar & H_*(\MU;\bZ) \dar \rar[hook] & H_*(\MU;\bQ) \dar[hook]\\
	 		 \pi_*\MSO \rar[hook] & H_*(\MSO;\bZ)  \rar & H_*(\MSO;\bQ)
  		\end{tikzcd}
	 	\]
	 	\caption{\label{fig:hurewicz} The Hurewicz homomorphisms with marked injections}
\end{figure}
\end{center}
\begin{proof}[Proof of \Cref{thm:vanishing}]
	First we show that $[M]$ is nilpotent if and only if $[M]$ is in the kernel of the integral Hurewicz map $h$. By \Cref{prop:char-classes}.\eqref{it:integral} we see that $H_*(MSp;\bZ)$ is reduced, i.e., there are no nilpotent elements. So if $[M]$ is nilpotent then $[M]\in \ker h$. For the converse we use the $E_\infty$ structure on $\MSp$ and apply the main result of \cite{MNN} which states that, under this hypothesis, every element in the kernel of $h$ is nilpotent. 

	Since $H_*(\MSp;\bZ)$ is torsion free, we see that if $[M]$ is torsion then $[M]\in \ker h$. If $[M]$ is not torsion then it has non-trivial image in $\pi_* \MSp\otimes \bQ \cong H_*(\MSp;\bQ)$. Since the rational Hurewicz map factors through the integral Hurewicz map (see \Cref{fig:hurewicz}), we see that $[M]$ can not be in $\ker h$. Since $[M]$ is torsion if and only if the second condition from \Cref{thm:vanishing} holds, we see that first two conditions of \Cref{thm:vanishing} are equivalent. 

	Now if we write $H_*(\MSp;\bZ)\cong \bZ[\overline{x}_{4i}]_{i\geq 1}$, then the Hurewicz image of $[M]$ has the following form \[ h[M]=\sum n_\alpha \overline{x}^\alpha. \] Here the sum is over all the monomials $\overline{x}^\alpha$ in the variables $\{\overline{x}_{4i}\}_{i\geq 1}$ whose total degree is the dimension of $M$. By a standard argument (see \cite[pp.~401-402]{Swi02}) one can show that \[n_\alpha=\langle P^\alpha(\nu), \sigma_M\rangle. \] Here $P^\alpha$ is the characteristic class dual to $x^\alpha$, $\nu$ is the stable normal bundle of $M$, and $\sigma_M$ is the fundamental class of $M$. These coefficients are, by definition, the Pontryagin numbers of $M$.  Clearly $h[M]=0$ if and only if these numbers vanish. So the first three conditions are equivalent.

	By naturality of the Hurewicz homomorphisms we see that if $h[M]=0$ then it has trivial image in $H_*(MU;\bZ)$ and $H_*(MSO;\bZ)$. Since $\pi_*MU$ is torsion free, the rational Hurewicz map $\pi_* MU\rightarrow H_*(MU;\bQ)\cong \pi_*MU\otimes \bQ$ is an injection which implies the integral Hurewicz map $\pi_* MU\rightarrow H_*(MU;\bZ)$ is an injection. Similarly $\pi_* MSO\rightarrow H_*(MSO;\bZ)$ is an injection because the composite map \[ \pi_*MSO\rightarrow H_*(MSO;\bZ)\rightarrow H_*(MSO;\bZ/2)\times H_*(MSO;\bQ)\] is an injection \cite[Cor.~1]{Mil60}. It follows that if $h[M]=0$ then $[M]$ has trivial image in the complex and oriented bordism rings (see \Cref{fig:hurewicz}).  Conversely if $[M]$ has trivial image in either of these bordism rings then it has trivial image in $H_*(MSO;\bQ)$. Since $H_*(\MSp;\bZ)$ injects into $H_*(MSO;\bQ)$ we see that $h[M]=0$.

\end{proof}

\begin{remark}
	The injectivity of the integral Hurewicz map for complex bordism is equivalent to the claim that a complex manifold is a boundary if and only if all of its Chern numbers vanish. Similarly an oriented manifold is a boundary if and only if its Stiefel-Whitney and Pontryagin numbers vanish. It follows from \Cref{thm:vanishing} that for \emph{stably symplectic} manifolds the vanishing of the Pontryagin numbers implies the vanishing of all of the Chern and Stiefel-Whitney numbers. Of course this can be independently verified via elementary arguments with characteristic classes.
\end{remark}

\begin{remark}\label{rem:nige}
We now show how the equivalence of conditions \ref{it:nil} and \ref{it:tor} in \Cref{thm:vanishing} can easily be deduced from the nilpotence theorem of \cite{DHS88}. The nilpotence theorem implies that that the kernel of the $\MU_*$-Hurewicz homomorphism \[ h_{MU}\colon \pi_* \MSp\rightarrow \MU_*\MSp\] is nilpotent. Since we already known that the torsion free summand of $\pi_*\MSp$ is non-nilpotent we just need to show that this kernel contains all of the torsion in $\pi_* \MSp$.  This will follow from the fact that $\MU_*\MSp$ is torsion free. Indeed the Atiyah-Hirzebruch spectral sequence \[ H_*(\MSp;\MU_*)\cong H_*(\MSp;\bZ)\otimes \MU_*\Longrightarrow \MU_*\MSp\] collapses for degree reasons. This is a free $\MU_*$-algebra and hence there are no extension problems.
\end{remark}
	\bibliographystyle{amsalphaabbrv}


	 \bibliography{symplectic.bbl} 
\end{document}